\documentclass[12pt]{amsart}
\usepackage[margin=1in]{geometry}
\usepackage{amsmath,amsfonts,graphicx}
\usepackage[utf8]{inputenc}
\usepackage{amsthm}
\usepackage{amssymb}
\usepackage{enumerate}
\usepackage{tikz-cd}
\usepackage{float}
\usepackage{hyperref}
\usepackage[capitalise]{cleveref}
\usepackage[utf8]{inputenc}
\usepackage[english]{babel}
\usepackage{lipsum}
\graphicspath{{./images/}}

\title{Subgroup Distortion and the Relative Dehn Functions of Metabelian Groups}
\author{Wenhao Wang}
\address{Department of Mathematics\\
  Vanderbilt University\\
 Nashville, TN 37240}
\email[W.~Wang]{wenhao.wang@vanderbilt.edu}

\newtheorem{theorem}{Theorem}[section]
\newtheorem{lemma}[theorem]{Lemma}
\newtheorem{corollary}[theorem]{Corollary}
\newtheorem*{corollaries}{Corollary}
\newtheorem{innercustomthm}{Theorem}
\newenvironment{customthm}[1]
  {\renewcommand\theinnercustomthm{#1}\innercustomthm}
  {\endinnercustomthm}

\theoremstyle{definition}
\newtheorem{proposition}[theorem]{Proposition}

\newtheorem{prob}[theorem]{Problem}

\theoremstyle{remark}

\newtheorem*{rems}{Remark}

\newcommand{\llangle}{\langle \langle}
\newcommand{\rrangle}{\rangle \rangle}

\DeclareMathOperator{\area}{Area}

\DeclareMathOperator{\supp}{supp}

\DeclareMathOperator{\rk}{rk}
\DeclareMathOperator{\OF}{OF}

\begin{document}
\maketitle

\begin{abstract}
  We show the connection between the relative Dehn function of a finitely generated metabelian group and the distortion function of a corresponding subgroup in the wreath product of two free abelian groups of finite rank. Further, we show that if a finitely generated metabelian group $G$ is an extension of an abelian group by $\mathbb Z$ the relative Dehn function of $G$ is polynomially bounded. Therefore, if $G$ is finitely presented, the Dehn function is bounded above by the exponential function up to equivalence. 
\end{abstract}

\section{Notation}
\label{notation}

Let $G$ be a group. For elements $x,y\in G, n\in \mathbb N$, our conventions are $x^{ny}=y^{-1}x^ny, [x,y]=x^{-1}y^{-1}xy$. We use double bracket $\llangle \cdot \rrangle_G$ to denote the normal closure of a set in the group $G$. Sometimes we omit the subscript when there is no misunderstanding in the context. 

In addition, for a group $G$ and a commutative ring $K$ with $1\neq 0$, we let $KG$ be the group ring of $G$ over $K$. An element $\lambda\in KG$ is usually denoted as $\lambda=\sum_{g\in G} \alpha_g g, \alpha_g\in K$ where all but finitely many $\alpha_g$'s are 0. We also regard $\lambda$ as a function $\lambda: G\to K$ with finite support, where $\lambda(g)=\alpha_g$. We let $|\lambda|=\sum_{g\in G} |\alpha_g|$.

For groups $A$ and $T$, let $B=\oplus_{t\in T}A^t$ be the direct sum of copies of $A$ indexed by elements in $T$. Then \emph{the wreath product} $A\wr T$ is defined to be the semidirect product $B\rtimes T$ where $T$ acts on $B$ by $t\circ(a_\omega)=(a_{t^{-1}\omega})$. The subgroup $B$ is called the base group of the wreath product. 

Suppose $G$ is an extension of $A$ by $T$ where both $A$ and $T$ are abelian. $A$ has a natural module structure over the group ring $\mathbb ZT$, and the action of $T$ on $A$ is given by conjugation. In this case, we also say that $G$ is an extension of a $T$-module $A$ by $T$. In this paper, we will use the following notation for actions of $\mathbb ZT$ on $A$. Let $\lambda=\sum_{t\in T} \alpha_t t \in \mathbb ZT$. Then for $a\in A$, we define
\[a^\lambda:=\prod_{t\in T}a^{\alpha_t t}.\]

\section{Introduction and Results}
\label{intro}

In this paper, we study two different functions associates with finitely generated metabelian groups, both of which are introduced to describe some geometric properties of groups, the (relative) Dehn function and the subgroup distortion function. They are naturally related since for a finitely presented group the Dehn function can be regarded as the distortion function of a subgroup in a free group over an infinite generating set (the set of all conjugates of relators). Asides from the canonical connection, we investigate the connection between the subgroup distortion function (of a special type of finitely generated subgroups) in the wreath product to the Dehn function relative to the variety of metabelian groups, which it interesting  because it provides an estimation of the usual Dehn function of a finitely presented metabelian group \cite{wang2020dehn}. 

We first define the (relative) Dehn function. Recall that a \emph{variety} of groups is a class of groups that closed under taking subgroups, epimorphic images, and unrestricted direct products. Inside a variety, we can talk about relative free groups and relative presentations, just like in the usual sense (the variety of all groups). Let $\mathcal V$ be a variety. The free group relative to the variety $\mathcal V$ on a set $X$, denoted by $\tilde F(X)$, is a group in $\mathcal V$ satisfying the following property: it is equipped with a map $\theta: X\to \tilde F(X)$ such that for every group $G$ in $\mathcal V$ and every set map $\sigma: X\to G$ there exists a unique homomorphism $\varphi: \tilde F(X)\to G$ such that the following diagram commutes,
\[
\begin{tikzcd}    
X \arrow[r,"\theta"] \arrow[rd,"\sigma"] &\tilde F(X) \arrow[d, "\varphi"]\\
&G
\end{tikzcd}.
\]
In particular, a group in $\mathcal V$ generated by $|X|$ elements is a epimorphic image of $\tilde F(X)$. A relative finite presentation $\mathcal P=\langle X\mid R\rangle_{\mathcal V}$ of $G$ consists of two finite sets $X$ and $R$ where $R$ is a subset of $\tilde F(X)$ such that there exists an epimorphism $\varphi:\tilde F(X)\to G$ and the kernel of $\varphi$ is $\llangle R\rrangle_{\tilde F(X)}$. Let $w$ be a word in $G$ such that $w=_G 1$. Then $w$ lies in the normal closure of $R$. Thus $w$ can be written as 
\[w=_{\tilde F(X)} \prod_{i=1}^l f_i^{-1}r_i{f_i} \text{ where }r_i\in R\cup R^{-1},f_i\in \tilde F(X).\]
The smallest possible $l$ is called the relative area of $w$, denoted by $\widetilde\area_\mathcal P(w)$. Consequently, the Dehn function relative to $\mathcal V$ with respect to the presentation $\mathcal P$ is defined as  
\[\tilde\delta_\mathcal P(n)=\sup\{\widetilde\area_\mathcal P(w)\mid |w|_{X}\leqslant n\}.\]
Here $|\cdot|_{X}$ is the word length in alphabet $X$. 

Throughout this paper, we focus on two varieties of groups: the variety of all groups and the variety of metabelian groups, where the later will be denoted by $\mathcal S_2$. 

If $\mathcal V$ is the variety of all groups, then the Dehn function relative to $\mathcal V$ is just the usual Dehn function. 

Dehn functions are defined up to an asymptotic equivalence $\approx$ taken on functions $\mathbb N \to \mathbb N$ by $f \approx g$ if and only if $f \preccurlyeq g$ and $g \preccurlyeq f$ where $f \preccurlyeq g$ if and only if there exists $C>0$ such that $f(n)\leqslant Cg(Cn)+Cn+C$ for all $n\in \mathbb N$. One can verify that $\approx$ is an equivalence relation. This relation preserves the asymptotic nature of a function. For example, it distinguishes polynomials of different degrees and likewise polynomials and the exponential function. It also distinguishes functions like $n^p$ and $n^p\log n$ for $p>1$. On the other hand, it identifies all polynomials of the same degree, and likewise all single exponentials, i.e., $a^n\approx b^n$ for $a,b>1$. 

Despite the dependence of Dehn function on finite presentations of a group, all Dehn functions of the same finitely presented group are equivalent under $\approx$ \cite{gromov1996geometric}, i.e., given a finitely presented group $G$ with finite presentations $\mathcal P$ and $\mathcal P'$, one can show that $\delta_{\mathcal P} \approx \delta_{\mathcal P'}$. Thus, we define the \emph{Dehn function} of a finitely presented group $G$, $\delta_G$, as the Dehn function of any of its finite presentation. 

If $\mathcal V$ is the variety of metabelian groups $\mathcal S_2$. It has been shown that the relative Dehn functions are also independent of the choice of finite presentations up to equivalence \cite{Fuh2000}. Therefore it is valid to denote the relative Dehn function of a finitely generated metabelian group $G$ by $\tilde \delta_G$. One non-trivial property of metabelian groups is that all finitely generated metabelian groups are relatively finitely presentable in $\mathcal S_2$ \cite{Hall1954}.

In what follows, the Dehn function of a finitely presented group $G$ and the area of a word $w$ in $G$ relative to the variety of all groups will be denoted by $\delta_G(n)$ and $\area(w)$ respectively, while the Dehn function of a finitely generated group $G$ and the area of a word $w$ in $G$ relative to the variety of metabelian groups will be denoted by $\tilde \delta_G(n)$ and $\widetilde \area(w)$ respectively.

Next, let us talk about the distortion function. Let $G$ be a finitely generated group with a finite generating set $X$ and $H$ be a a subgroup of $G$ with finite generating set $Y$. The \emph{distortion function} of $H$ in $G$ is 
\[\Delta_H^G(n)=\sup\{|w|_Y\mid w\in H, |w|_X\leqslant n\}.\]
We consider a slightly different equivalence relation for distortion functions. For non-decreasing functions $f$ and $g$ on $\mathbb N$, we say that $f\preceq g$ if there exists a constant $C$ such that $f(n)\leqslant Cg(Cn)$. Hence we say that two functions $f$ and $g$ are equivalent, written $f\asymp g$, if $f\preceq g$ and $g\preceq f$. As expected, the distortion function is independent of the choice of the finite generating set under this equivalence relation \cite[Proposition 8.98]{drute2018Geometric}. The reason we consider $\asymp$ rather than $\approx$ is that if the subgroup is infinite then the distortion function is at least linear. We say a subgroup is \emph{undistorted} if the distortion function is equivalent to a linear function.

Let $A$ and $T$ be free abelian groups with bases $\{a_1,a_2,\dots,a_m\}$ and $\{t_1,t_2,\dots,t_k\}$ respectively. Consider the wreath product $W:=A\wr T$. The base group $B:=\llangle A\rrangle$ is a $T$-module. For a finite subset $\mathcal X=\{f_1,f_2,\dots,f_l\}$ of $B$, let $H$ be the subgroup of $W$ generated by $\mathcal X\cup\{t_1,t_2,\dots,t_k\}$ and $G$ be the group $W/\llangle \mathcal X\rrangle$. 

Our main result is the following:
\begin{customthm}{A}[\cref{subgroupDistortion}]
	\label{subgroupDistortion}
	Let $W,H,G$ be groups defined as above, then 
	\[\Delta_{H}^W(n) \preccurlyeq \tilde \delta_G^k(n)+n^k,\tilde\delta_G(n)\preccurlyeq\max\{n^3, (\Delta_{H}^W(n^2))^3\}.\]
	In particular, if $k=1$, 
	\[\Delta_{H}^W(n) \preccurlyeq \tilde \delta_G(n).\]
\end{customthm}

\cref{subgroupDistortion} leads to some interesting examples.

\begin{corollaries}[\cref{relativeLowerbound}]
For each $l\in \mathbb N$, there exists a finitely generated metabelian group such that its relative Dehn function is asymptotically greater or equal to $n^l$.
\end{corollaries}

The distortion function of subgroups in $A\wr \mathbb Z$ has been studied extensively by Davis and Olshanskiy \cite{davis2011Subgroup}. Combining their result with \cref{subgroupDistortion}, we immediately have

\begin{customthm}{B}[\cref{overZ}]
	Let $G$ be a finitely generated metabelian group such that $G$ is an extension of an abelian group $A$ by a virtually cyclic abelian group $T$. Then the relative Dehn function of $G$ is polynomially bounded. If in addition $G$ is finitely presented, the Dehn function of $G$ is asymptotically bounded above by the exponential function.
\end{customthm}

This theorem gives the exponential upper bound of Dehn functions for many examples including the metabelian Baumslag-Solitar groups $BS(1,n)$ and $\mathbb Z^n\rtimes_\phi \mathbb Z$ where $\phi\in GL(n,\mathbb Z)$ (See in \cite{bridson1996optimal}). And it also improves the main results in \cite{wang2020dehn}.

Moreover, we estimate the relative Dehn function of various examples.

\begin{customthm}{C}
	\begin{enumerate}[(1)]
		\item (\cref{metaBaum}) The metabelianized Baumslag-Solitar group $\tilde{BS}(n,m)=\langle a,t\mid (a^{n})^{t}=a^m\rangle_{\mathcal S_2}$ has at most cubic relative Dehn function when $n\neq m$ and has at most quartic relative Dehn function when $n=m$.
		\item (\cref{metaBaumCorollary}) The metabelianized Baumslag-Solitar group $\tilde{BS}(n,m)=\langle a,t\mid (a^{n})^{t}=a^m\rangle_{\mathcal S_2}, m>2, m=n+1$ has at most quadratic relative Dehn function.
		\item (\cref{lamplighter}) The lamplighter groups $L_m$ have at most cubic relative Dehn function.
		\item (\cref{relativeL2}) The lamplighter group $L_2$ has linear relative Dehn function.
	\end{enumerate}
\end{customthm}

\emph{The structure of this paper.} In \cref{prem} we will state some preliminaries on the topic of Dehn function of a module and the relative Dehn function of a finitely generated metabelian group. Next in \cref{relativeDehn4}, we estimate the relative Dehn function for different examples including the Baumslag-Solitar groups and Lamplighter groups. Finally in \cref{relativeDehn5} we prove the main theorem about distortion function and the relative Dehn function.

\emph{Acknowledgement.} I would like to thank my advisor Mark Sapir who points out to me the study of distortion functions for subgroups in $A\wr \mathbb Z$ where $A$ is abelian.

\section{Preliminaries}
\label{prem}

\subsection{The Dehn function of a Module}
\label{prem1}

Let $T$ be a free abelian group of rank $k$ and $R=\mathbb ZT$ the group ring of $T$. In what follows, we will only discuss modules over $R$.

Similar to groups, we have free modules and hence we can define the presentation of a module. A subset $\{f_1,f_2,\dots,f_l\}$ of a $R$-module $M$ is called a \emph{generating set} if every $f\in M$ is the linear combination of them, i,e, there exists $\alpha_1,\alpha_2,\dots,\alpha_l\in R$ such that 
\[f=\alpha_1 f_1+\alpha_2 f_2+\dots+\alpha_l f_l.\]

A set of elements $\{f_1,f_2,\dots,f_l\}$ of a module $M$ is called \emph{independent} if no nontrivial linear combination is zero, that is,

\[\text{If } \alpha_1 f_1+\alpha_2 f_2+\dots+\alpha_l f_l=0, \text{ then }\alpha_i=0, \text{ for } i=1,2,\dots,l.\]

A \emph{basis} is an independent generating set. 

One immediate example for a $R$-module is $R^m$. The addition and scalar multiplication on $R^m$ are the following, respectively:
\begin{align*}
	(a_1,a_2,\dots,a_m)+(b_1,b_2,\dots,b_m)&=(a_1+b_1,a_2+b_2,\dots,a_m+b_m),\\
	r(a_1,a_2,\dots,a_m)&=(ra_1,ra_2,\dots,ra_m).
\end{align*}
The module $R^m$ is called a \emph{free $R$-module of rank $m$}. The canonical basis of $R^m$ is $\{e_1,e_2,\dots,e_m\}$ where $e_i=(0,\dots,1,\dots,0)$ with all but the $i$-th entry is 0. 

A submodule of the free module $R^1$ is an ideal in the ring $R$.

Given a free $R$-module $M$ of finite rank and a submodule $S$ generated by a finite set $\{f_1,f_2,\dots,f_l\}$, the membership of a submodule $S$ we are considering in this thesis is the following

\begin{prob}
\label{memberPro}
Given an element $f$ in $M$, decide whether $f$ in $S$, i.e., if there exists elements $\alpha_1,\alpha_2,\dots,\alpha_l$ such that 
\[f=\alpha_1f_1+\alpha_2f_2+\dots+\alpha_l f_l.\]
\end{prob}

A \emph{homomorphism} $\varphi: M \to N$ of $R$-modules is a map which is compatible with the laws of composition:

\[\varphi(f+f')=\varphi(f)+\varphi(f'), \varphi(rf)=r\varphi(f)\]

for all $f,f'\in M, r\in R$. A bijective homomorphism is called an \emph{isomorphism}. 

Last we define the concept of quotient modules. Let $R$ be a ring, and let $S$ be a submodule of an $R$-module $M$. The quotient $M/S$ is the additive group of cosets $\bar f=f+S$. And the scalar multiplication is defined by 
\[r\bar f=\overline{rf}.\]
Thus $M/S$ is made an $R$-module. 

Conversely, let $A$ be a finitely generated $R$-module, then there exists a free $R$-module $M$ with basis $\{a_1,a_2,\dots,a_m\}$ and a submodule $S$ of $M$ such that $A\cong M/S$. Since $R$ is a Noetherian ring, $S$ is finitely generated, and hence we assume its generating set is $\{f_1,f_2,\dots,f_l\}$. Therefore we have a module presentation of $A$ as the following:
\[A=\langle a_1,a_2,\dots,a_m\mid f_1,f_2,\dots, f_l\rangle.\]

For an element $f=\mu_1 a_1+\mu_2 a_2+\dots +\mu_m a_m$ in $M$ we define its \emph{length}, denoted by $\|f\|$, to be the following:
 \[\|f\|=\sum_{i=1}^l |\mu_i|+\mathrm{reach}(f),\]
 where $\mathrm{reach}(f)$ is the minimal length over the lengths of close loops that starts at $1$ and passes through all points in $\cup_{i=1}^l \supp{\mu_i}$ in the Cayley graph of $T$. Another way to think of this length $\|\cdot\|$ is that it is the minimal length of a group word among words that are rearranges of all conjugates of elements in $\mathcal A$ in $a_1^{\mu_1}a_2^{\mu_2}\dots a_m^{\mu_m}$. For example, suppose $m=k=1$, we have 
 \[\|(t_1^n+t_1^{n-1}+\dots+t_1+1)a_1\|=(n+1)+2n=3n+1,\]
 because the minimal length of a loop passing $\{1,t,t^2,\dots,t^n\}$ is $2n$. Note that $a_1^{t_1^n+t_1^{n-1}+\dots+t_1+1}=t_1^{-n}a_1t_1a_1\dots a_1t_1a_1$ is a group word of length $3n+1$ in the alphabet $\{a_1\}\cup\{t_1\}$.
  
 Then for every element $f$ in $S$, there exists $\alpha_1,\alpha_2,\dots,\alpha_l\in R$ such that 
\[f=\alpha_{1}f_1+\alpha_2 f_2+\dots+\alpha_l f_l.\]
We denote by $\widehat\area_A(f)$ the minimal possible $\sum_{i=1}^l |\alpha_i|$ ($|\cdot|$ is defined in \cref{notation}). Then \emph{the Dehn function} of the $R$-module $A$ is defined to be 
\[\hat\delta_A(n)=\max\{\widehat \area_A(f)\mid \|f\|\leqslant n\}.\]
As expected, the Dehn function of a module is also independent from the choice of the finite presentation \cite{Fuh2000}. 
\begin{rems}
	Now we have three different types of Dehn functions in this paper: the Dehn function, the relative Dehn function and the Dehn function of a module. They are similar and we distinguish them by the notation. We denote by $\delta_G(n), \area(w)$ the Dehn function of $G$ and the area of a word $w$; $\tilde \delta_G(n),\widetilde \area(w)$ the relative Dehn function and the relative area of a word $w$, $\hat \delta_A(n), \widehat \area(f)$ the Dehn function of the module $A$ and the area of a module element $f$.
\end{rems}

The membership problem \cref{memberPro} can be regarded as the word problem of the quotient $M/S$.

\begin{prob}
\label{memberPro}
Given an element $f$ in $A$, decide whether $f$ represents the trivial element in $A$.
\end{prob}

It turns out that the Dehn function of a module plays an essential role for understanding the relative Dehn function for finitely generated metabelian group \cite{Fuh2000}. 

Last, let us define a well-order $\prec$ on the ring $R=\mathbb ZT$. On $\mathbb Z$, we define an order $\prec_{\mathbb Z}$ as following
\[0\prec_{\mathbb Z} 1 \prec_{\mathbb Z} 2 \prec_{\mathbb Z} \dots \prec_{\mathbb Z} -1 \prec_{\mathbb Z} -2 \prec_{\mathbb Z} \dots.\]
For monomials in $R$, we use the degree lexicographical order (also called shortlex or graded lexicographical order) $\prec_R$ which is defined with respect to the convention $t_1 \succ t_1^{-1} \succ \dots \succ t_k\succ t_k^{-1}$, i.e. for $\mu_1=t_1^{n_1}t_1^{-n_2}t_2^{n_3}t_2^{-n_4}\dots t_k^{n_{2k-1}}t_k^{-n_{2k}}, \mu_2=t_1^{m_1}t_1^{-m_2}t_2^{m_3}t_2^{-m_4}\dots t_k^{m_{2k-1}}t_k^{-m_{2k}}$
\[\mu_1\prec_R \mu_2 \text{ if }\sum_{i=1}^{2k}|n_i|<\sum_{i=1}^{2k} |m_i| \text{ or } \sum_{i=1}^{2k}|n_i|=\sum_{i=1}^{2k} |m_i|, \mu_1\prec_{lex} \mu_2,\]
where $\prec_{lex}$ is the usual lexicographical order which is defined in the following way
\[t_1^{n_1}t_1^{-n_2}t_2^{n_3}t_2^{-n_4}\dots t_k^{n_{2k-1}}t_k^{-n_{2k}}\prec_{lex} t_1^{m_1}t_1^{-m_2}t_2^{m_3}t_2^{-m_4}\dots t_k^{m_{2k-1}}t_k^{-m_{2k}}\]
if $n_i<m_i$ for the first $i$ where $n_i$ and $m_i$ differ. Note that $\prec_R$ on $\mathcal X$ in fact is a well-oder while $\prec_{lex}$ might not be (See in  \cite{baader1999term}).

Finally we set $\prec$ on $R=\mathbb ZT$ to be the lexicographical order based on $T \succ \mathbb Z$. It is not hard to verify that $\prec$ is a well-order on $\mathcal T$. The degree $\deg \mu$ of $\mu\in \mathbb ZT$ is define to be the degree of its leading monomial and the degree $\deg f$ of $fin M$ is define to be the maximal degree of coefficients of basis.

\subsection{Relative Dehn Functions of Finitely Generated Metabelian Groups}
\label{prem2}

As we discuss above, the Dehn function for a finitely generated metabelian group relative to $\mathcal S_2$ is always defined. So it provides a convenient tool to study finitely generated metabelian groups. 

Let $G$ be a finitely generated metabelian group. Then $G$ sits inside a short exact sequence. 
\[1\to A\to G\to T\to 1,\]
where $A$ and $T$ are abelian. Since $G$ is finitely generated, $T$ is finitely generated abelian group. We now choose $A$, $T$ among all such short exact sequences such that the torsion-free rank of $T$ is minimized. We denote by $\rk(G)$ the minimal torsion-free rank of $T$. 

Now let $k=\rk(G)$ and $\pi:G\to T$ be the canonical quotient map. It is not hard to show that there exists a subgroup $G_0$ of finite index in $G$ such that $\pi(G_0)=\mathbb Z^k$. It has been shown that the Dehn function, the relative Dehn function and $\rk(G)$ are all preserved (up to equivalence) under taking finite index subgroups \cite{gromov1996geometric}, \cite{wang2020dehn}. Thus, in what follows, we always assume that $T$ is a free abelian group.

For the finitely generated metabelian group $G$, let $\{t_1,t_2,\dots,t_k\}$ be a subset of $G$ such that its image in $T$ form a basis. Since $A$ is a normal subgroup in $G$, by \cite{Hall1954}, it is a normal closure of a finite set. Let $\mathcal A=\{a_1,a_2,\dots,a_m\}$ be a subset of $A$ satisfying: (1) $\llangle \mathcal A\rrangle=A$; (2) $\mathcal A$ contains all commutators of $\{t_1,t_2,\dots,t_k\}$, i.e., for any pair $1\leqslant i<j\leqslant k$, there exists $l(i,j)\in \{1,2,\dots,m\}$ such that $a_{l(i,j)}=[t_i,t_j]$, $l(i,j)=l(i',j')$ if and only if $i=i',j=j'$.
 
We associate $G$ with an auxiliary group $\tilde G$, which has the following relative presentation:
\begin{align*}
	\tilde G=&\langle a_1,a_2,\dots,a_m,t_1,t_2,\dots,t_k\mid a_{l(i,j)}=[t_i,t_j], \\
	&[a,b]=1,[a,b^t]=1, 1\leqslant i<j\leqslant k, a,b\in \mathcal A, t\in \mathcal T\rangle_{\mathcal S_2}
\end{align*}
Relations $\{[a,b]=1,[a,b^t]=1, a,b\in \mathcal A, t\in \mathcal T\}$ is enough along with all metabelian relations inherited from the free metabelian groups. And the area of $[a,b^u]$ is linearly controlled by the length of $u$, that is,

\begin{lemma}[Wang \cite{wang2020dehn}]
\label{relativeCommutative}
	$\{[a,b]=1,[a,b^t]=1, a,b\in \mathcal A, t\in \mathcal T\}$ generates all commutative relations $[a,b^u]=1, a,b\in \mathcal A, u\in F(\mathcal T)$ in the presentation relative to the variety of metabelian groups. Moreover, the relative area of $[a,b^u]$ is bounded by $4|u|-3$.
\end{lemma}

All relations in $\tilde G$ are also represent the identity in $G$. It follows that the identity map on $\mathcal A\cup \mathcal T$ induces an epimorphism $\varphi: \tilde G\to G$. The kernel $\ker\varphi$ is a normal subgroup in $\llangle \mathcal A\rrangle_{\tilde G}$, since $\varphi$ induces an isomorphism on $T$. Let $\{f_1,f_2,\dots,f_l\}$ be a subset of $\llangle \mathcal A\rrangle_{\tilde G}$ such that $\llangle f_1,f_2,\dots,f_l\rrangle_{\tilde G}=\ker \varphi$. Thus we obtain a relative presentation of $G$. 
\begin{align*}
	G=&\langle \tilde G\mid f_1,f_2,\dots,f_l\rangle_{\mathcal S_2}\\
	=&\langle a_1,a_2,\dots,a_m,t_1,t_2,\dots,t_k\mid f_1,f_2,\dots,f_l, a_{l(i,j)}=[t_i,t_j],\\
	&[a,b]=1,[a,b^t]=1, 1\leqslant i<j\leqslant k, a,b\in \mathcal A, t\in \mathcal T\rangle_{\mathcal S_2}.
\end{align*}

We focus on the module structure on $\llangle \mathcal A\rrangle_{\tilde G}$, which is a free $T$-module generated by the basis $\mathcal A$ \cite{wang2020dehn}. Let us define the ordered form of an element $f$, denoted by $\OF(f)$, in $\llangle \mathcal A\rrangle_{\tilde G}$. $f$ can be written as $a_1^{\alpha_1}a_2^{\alpha_2}\dots a_m^{\alpha_m}$ as a group element or $\alpha_1 a_1+\alpha_2 a_2+\dots+\alpha_m a_m$ as a module element. Let $\prec$ be the well-order we construct in \cref{prem1}. The ordered form $\OF(f)$ is of the form $a_1^{\mu_1}a_2^{\mu_2}\dots a_m^{\mu_m}$ such that 
	\begin{enumerate}[(1)]
	    \item $\mu_i\in \mathbb ZT$ for $1\leqslant i\leqslant m$, and each $\mu_i$ is of the form $\mu_i=\sum_{j=1}^{n_j} c_{ij}u_{ij}$ such that $c_{ij}\in \mathbb Z, u_{ij}\in \bar F$ and $u_{i1}\succ u_{i2}\succ \dots \succ u_{in_i}$;
		\item $f=_{\tilde G} \OF(f)$,
	\end{enumerate} 
	where 
	\[\bar F=\{t_1^{m_1}t_2^{m_2}\dots t_k^{m_k}\mid m_1,\dots,m_k\in \mathbb Z\}.\]
	It has been shown that the ordered form is well-defined and for any $f,g\in \llangle \mathcal A\rrangle_{\tilde G}$, $f=_{\tilde G}g$ if and only if $\OF(f)=_{F(\mathcal A\cup \mathcal T)} \OF(g)$ \cite{wang2020dehn}. For an element $f$ in $\llangle \mathcal A\rrangle_{G}$, we define the ordered form of $f$ by lifting $f$ to $\tilde G$. The ordered form is useful for estimating the relative area of a word.
	
Let $M$ be the free $T$-module with basis $\{a_1,a_2,\dots,a_m\}$ and $S$ be a submodule of $M$ generated by $\{f_1,f_2,\dots,f_l\}$. Then the $T$-module $A$ is isomorphic to $M/S$. We have the following connection between the Dehn function of the module $A$ and the relative Dehn function of $G$:

\begin{proposition}[Wang \cite{wang2020dehn}]
\label{relativeConnection1}
	Let $G$ be a finitely generated metabelian group and $A$ is defined as above, then 
	\[\hat\delta_A(n)\preccurlyeq \tilde \delta_G(n) \preccurlyeq \max\{\hat\delta_A^3(n^3), n^6\}. \]
\end{proposition}

If $G$ is a semidirect product, the result can be slightly improved.

\begin{proposition}
\label{improved}
	Let $T$ be a finitely generated abelian group and let $A$ be a finitely generated $T$-module. Form the semidirect product
	\[G=A\rtimes T.\]
	Then $\delta_G(n)\preccurlyeq \max\{\hat\delta_A^3(n^2), n^3\}.$	
\end{proposition}

If $G$ happens to be finitely presented, we have

\begin{theorem}
\label{relativeConnection2}
Let $G$ be a finitely presented metabelian group. Then 
	\[\tilde \delta_G(n)\preccurlyeq \delta_G(n)\preccurlyeq \max\{\tilde \delta_G^3(n^3),2^n\}.\]
\end{theorem}
	
\section{Estimate the Relative Dehn Function}
\label{relativeDehn4}

Computing the relative Dehn function is harder than computing the Dehn function. Many techniques are no longer useful for the relative case. For the variety of metabelian groups, fortunately, the structure of groups in it is not complicated. The key is to understand the natural module structure of a finitely generated metabelian group. 

First, let us list some known results for relative Dehn functions, they are computed by Fuh in her thesis \cite{Fuh2000}. Note that most of them only give the upper bound of the relative Dehn function.

\begin{theorem}[Fuh \cite{Fuh2000}]
\label{knownResult}
	\begin{enumerate}[(1)]
		\item The realative Dehn function of a wreath product of two finitely generated abelian groups is polynomially bounded.		
		\item The Baumslag-Solitar group $BS(1,2)$ has linear Dehn function.
		\item Let $G=\tilde{BS}(n,m)=\langle a,t \mid (a^n)^t = a^m \rangle_{\mathcal S_2}$ where $m>2, m=n+1$. Then $\tilde\delta_G(n)\preccurlyeq n^3$.
	\end{enumerate}
\end{theorem}

Now let us estimate the relative Dehn function from above for some concrete examples. By the \emph{cost} of converting $w_1$ to $w_2$ ($w_2$ to $w_1$) in $G$ we mean the relative area of $w_2^{-1}w_1$ (resp. $w_1^{-1}w_2$) in $G$. If $w_2$ happens to be the identity, then the cost of converting $w_1$ to $w_2$ coincides with the area of $w_1$. By the definition of the area, it is not hard to see that if $w_1=_G w_2 =_G w_3$ and the cost of converting $w_1$ to $w_2$, $w_2$ to $w_3$ is $N_1$ and $N_2$ respectively then the cost of converting $w_1$ to $w_3$ is at most $N_1+N_2$. Essentially, to estimate the relative Dehn function from above we need to estimate the cost of converting a word to the identity. 

To begin with, we consider the metabelianized Baumslag-Solitar group 
\[\tilde{BS}(n,m)=\langle a,t\mid (a^n)^t=a^m\rangle_{\mathcal S_2}.\]
The normal subgroup generated by $a$ is a $\mathbb Z\langle t\rangle$-module. In this case, i.e., when the module is over the ring of Laurent polynomial of one variable and is generated by one variable, the Dehn function of the module is well-studied. The following theorem from Davis and Olshanskiy \cite{davis2011Subgroup} shows that the Dehn function of a finitely generated $\langle t \rangle$-module is a polynomial. 

\begin{theorem}[Davis, Olshanskiy {\cite[Theorem 8.6]{davis2011Subgroup}}]
\label{moduleoverone}
	Let $M=\langle a\rangle$ is the free module of rank one over the group ring $\mathbb Z\langle t\rangle$. Let $f=h(t)a$ where $h(x)$ is a polynomial of the form $d_nx^n+d_{n-1}x^{n-1}+\dots+d_0$. Then the Dehn function of the $\langle t\rangle$-module $M/\langle f\rangle$ is a polynomial. Furthermore, the degree of this polynomial is exactly one plus the maximal multiplicity of a (complex) root of $h(x)$ having modulus one.
\end{theorem}

Thus we have
 
\begin{proposition}
\label{metaBaum}
The metabelianized Baumslag-Solitar group $\tilde{BS}(n,m)=\langle a,t\mid (a^{n})^{t}=a^m\rangle_{\mathcal S_2}$ has at most cubic relative Dehn function when $n\neq m$ and has at most quartic relative Dehn function when $n=m$.
\end{proposition}

\begin{proof}
	Note that in this case we have $|\mathcal A|=|\mathcal T|=1$, which simplifies the process a lot. Given a word $w=_G 1$ of length $l$. 
	
	First we claim that converting $w$ to the ordered form $\OF(w)$ (defined in \cref{prem2}) costs at most $(4l-3)l^2$. 
	
	Suppose $w=t^{n_1}a^{m_1}t^{n_2}a^{m_2}\dots t^{n_s}a^{m_s}t^{n_{s+1}}, n_i, m_i\in\mathbb Z$ for all $i$ and only $n_1,n_{s+1}$ can be zero. Since $w=1$, we have that 
	\[\sum_{i=1}^{s+1} n_i=0, \sum_{i=1}^{s}(|n_i|+|m_i|)+|n_{s+1}|=l.\]
	The first conditions comes from the fact that the image of $w$ is 0 in $T$ and the second condition comes from $|w|=l$.
	
	Thus we can rewrite $w$ by inserting trivial words $tt^{-1}$ to the form 
	\[w=a^{t^{-n_1}}a^{t^{-(n_1+n_2)}}\dots a^{t^{-(n_1+n_2+\dots+n_s)}}=a^\mu,\]
	where $\mu=t^{-n_1}+t^{-(n_1+n_2)}+\dots+t^{-(n_1+n_2+\dots+n_s)}\in \mathbb ZT.$ We immediately have that $\deg(\mu)\leqslant l$, $|\mu|=s\leqslant l$, and $\|w\|\leqslant l$ by the definition. 
	
	The cost of converting $w$ to $a^{t^{-n_1}}a^{t^{-(n_1+n_2)}}\dots a^{t^{-(n_1+n_2+\dots+n_s)}}$ is zero since we only insert trivial words in the absolute free group. 
	
	Next we will convert $a^{t^{-n_1}}a^{t^{-(n_1+n_2)}}\dots a^{t^{-(n_1+n_2+\dots+n_s)}}$ to the ordered form of $w$. To do this, we have to rearrange conjugates of $a$ such that exponents are ordered by $\prec$ from high to low. In order to commute two conjugates in $w$, we have to insert commutator of the form
	\[[a^{t^{-l_i}}, a^{t^{-l_j}}]=[a,a^{t^{l_i-l_j}}], \text{ where }l_i=n_1+n_2+\dots+n_i.\]
	By \cref{relativeCommutative}, the area is bounded by $4l-3$. To rearrange $s$ conjugates of $a$ we need to insert at most $s^2$ many such commutators. Thus the cost of converting $w$ to $\OF(w)$ is bounded by $(4l-3)l^2$. The claim is proved.
	
	Suppose $OF(w)=a^{\mu}$, where $|\mu|\leqslant l,\deg \mu\leqslant l$. We can conjugate $w$ by $t^l$ such that $\mu$ only have positive powers of $t$. Thus we assume that $|\mu|\leqslant l,\deg \mu \leqslant 2l$. Further, the length of $\mu$ is bounded by $l$ by definition.
	
	In this case, the module $A$ is isomorphic to $M/S$ where $M$ is a free $T$-module with basis $a$ and $S$ is its submodule generated by $\{(nt-m)a\}$. Consider the polynomial ring $R=\mathbb Z[t,t^{-1}]$ and its ideal $I=\langle nt-m,tt^{-1}-1\rangle$. We have that $A\cong R/I$. The Gr\"{o}bner basis of $I$ is $\{tt^{-1}-1, nt-m, mt^{-1}-n\}$. If we regard $\mu$ as an element in $I$, it can only be reduced by $nt-m$ since it only has positive power of $t$. It follows that there exists a polynomial $\nu$, which only consists of the power of $t$, such that 
	\[\mu=(nt-m)\nu.\]
	This equality also holds in the polynomial ring $\mathbb Z[t]$. When $n\neq m$, the Dehn function of $\langle t\rangle$-module $\mathbb ZT/\langle nt-m\rangle$ is linear, by \cref{moduleoverone}. Thus there exists $C$ such that $|\nu|\leqslant C\|\mu\|+C$. We have that 
	\[a^\mu=_G(a^{mt-n})^{\nu}.\]
	The area of the right hand side is at most $Cl+C$. Converting the right hand side to its ordered form costs at most $(4l-3)((m+n)(Cl+C))^2$ since the degree is less than $l$ and we have $(m+n)(Cl+C)$ many conjugates to rearrange. Thus the upper bound of $\widetilde \area(w)$ is at most $l^3$ up to equivalence when $n\neq m$.
	
	When $n=m$, the Dehn function of $\langle t\rangle$-module $\mathbb ZT/\langle nt-m\rangle$ is quadratic. Following the same process, we have that the upper bound of $\widetilde \area(w)$ is at most $l^4$ up to equivalence when $n\neq m$. This finishes the proof.
\end{proof}

For the case $n=1$, the group $\tilde BS(1,n)\cong BS(1,n)$ is finitely presented. Following from \cref{relativeConnection2}, the Dehn function of $BS(1,n)$ is at most exponential. We will extend this idea of using relative Dehn function to estimate the Dehn function in the next Section.

One special case Fuh \cite[Theorem 6.1]{Fuh2000} concerned is when $m>2, m=n+1$. In this case, we have that $a=[a^n,t]$. Since $a$ itself is a commutator, it follows that the relative area of words like $[a^{t^k},a]$ is at most 4 instead of linearly depending on $k$. Therefore we can improve the result in \cite[Theorem 6.1]{Fuh2000} by the following corollary of \cref{metaBaum}.

\begin{corollary}
\label{metaBaumCorollary}
	The metabelianized Baumslag-Solitar group $\tilde{BS}(n,m)=\langle a,t\mid (a^{n})^{t}=a^m\rangle_{\mathcal S_2}, m>2, m=n+1$ has at most quadratic relative Dehn function.
\end{corollary}

The lamplighter groups are another interesting class of infinite presented metabelian groups with a simple module structure. We have

\begin{proposition}
\label{lamplighter}
	The lamplighter groups $L_m, m\geqslant 2$ have at most cubic relative Dehn function. 
\end{proposition}

\begin{proof}
	Consider the lamplighter group $L_m$ with the standard presentation. 
	\[L_m=\langle a,t\mid a^m=1, [a,a^{t^n}]=1, n\in\mathbb N\rangle.\]
	By the discussion in \cref{prem2}, we have a finite relative presentation as the following
	\[L_m=\langle a,t\mid a^m=1, [a,a^t]=1\rangle_{\mathcal S_2}.\]
	The rest of the proof is the same as the proof of \cref{metaBaum}. The only difference is that in this case the submodule is generated by $\{m\}$.
\end{proof}

This slightly improves the estimation in \cite[Theorem B2]{Fuh2000}.

In particular, when $m=2$, for the case of $L_2$, we are able to improve the upper bound to linear. 

\begin{proposition}
\label{relativeL2}
	The lamplighter groups $L_2$ has linear relative Dehn function.
\end{proposition}

\begin{proof}
	The linear lower bound is given by \cref{moduleoverone}. 
	
	We choose the following relative presentation of $L_2$:
	\[L_2=\langle a,t\mid a^2=1, [a,a^t]=1\rangle_{\mathcal S_2}.\]
	For the upper bound, consider a word $w\in L_2$ that represents the identity. Thus $w$ has the form 
	\[w=t^{n_1}at^{n_2}a^{n_3}\dots t^{n_{2k}}at^{n_{2k+1}}, \text{ where }n_2,n_3,\dots,n_{2k}\neq 0.\]
	Suppose the length of $w$ is $n$, combining the fact that $w=1$, we have
	\[2k+\sum_{i=1}^{2k+1}|n_i|=n, \sum_{i=1}^{2k+1} n_i=0.\]
	Inserting $tt^{-1}$ or $t^{-1}t$, we can rewrite $w$ as the following form:
	\[w=a^{t^{-n_1}}a^{t^{-(n_1+n_2)}}\dots a^{t^{-(n_1+n_2+\dots+n_{2k})}}.\]
	Thus $w$ represents an element in $\oplus_{i\in \mathbb Z}\mathbb Z_2$, where the $a^{t^i}$ is the generator of the $i$-th copy of $\mathbb Z_2$. Since $w=1$, then every element in the set $\{-n_1,-(n_1+n_2),\dots,-(n_1+n_2+\dots+n_{2k})\}$ occurs even many times in the sequence $-n_1,-(n_1+n_2),\dots,-(n_1+n_2+\dots+n_{2k})$. Our goal is to gather the conjugates of $a$ of the same exponents together at a linear cost with respect to $n$.
	
	Since $a^{-1}=a$, we notice that 
	\[a^{t^s}a^{t^l}=(aa^{t^{l-s}})^{t^s}=[a,t^{l-s}]^{t^s}, l,s\in\mathbb Z.\]
	Thus any pair of two consecutive conjugates of $a$ is a commutator. It follows that any such pair commutes with any other pair of this form without any cost inside the variety of metabelian groups.
	
	For convenience, let $m_i=\sum_{i=1}^{2k} -n_{i}$. We now turn the problem of estimating the relative area of $w$ to a problem of cancelling numbers in a sequence and estimating the cost. Consider a sequence of number 
	\[m_1,m_2,\dots,m_{2k}.\]
	The goal is to cancel all the pairs of the same value. We have three operations allowed:
	\begin{enumerate}[(i)]
		\item Cancel two consecutive numbers of the same value without any cost. 
		\item Commute a pair of consecutive numbers with another pair of consecutive numbers next to it without any cost.
		\item Commute two consecutive numbers $c,d$ with a cost of $|c-d|$.
	\end{enumerate}
	
	Applying all three operations to the original sequence many times, the result might seems chaotic. To analyze the process, for a sequence of numbers, we define the $\iota(m_i)$ be the position of $m_i$ in the sequence. At the beginning, $\iota(m_i)=i$. Then we define $\sigma(m_i,m_j)=|\iota(m_i)-\iota(m_j)| \mod 2$. So $\sigma(m_i,m_j)=0$ if $m_i$ and $m_j$ are even positions apart and $\sigma(m_i,m_j)=1$ if $m_i$ and $m_j$ are odd positions apart. We notice that 
	\begin{enumerate}[(a)]
		\item operations from (i) and (ii) do not change $\sigma(m_i,m_j)$;
		\item if $m_i$ is next to $m_j$, applying the operation (iii) to commute $m_i$ and $m_j$ will change all values of $\sigma(m_i,m_l), \sigma(m_j,m_l)$ for $l\neq i ,j$ but all other values of $\sigma$ remain the same. 
	\end{enumerate}
	
	From the above observation, we have that 
	\begin{enumerate}[(1)]
		\item if $\sigma(m_i,m_j)=0$ and $i<j$, $m_j$ can be moved to the position next to $m_i$ just using operations from (ii).
		\item if $\sigma(m_i,m_j)=0$, $i<j$ and $m_i=m_j$, then $m_i$ and $m_j$ can be cancelled using just operations from (i) and (ii).
		\item for $m_i,m_j,m_l$ such that $m_i=m_j$, $\sigma(m_i,m_j)=1,\sigma(m_i,m_l)=0$, we can cancel $m_i,m_j$ with the cost of $|m_i-m_l|$.
	\end{enumerate}
	(1) can be achieved by commuting two consecutive pairs of numbers. (2) is a direct consequence of (1). Let us show how to achieve (3). By (1), we can move $m_l$ next to $m_i$. Then by using operation (ii), the pair $m_im_l$ (or $m_lm_i$) can be moved to the position next to $m_j$, resulting the form of $m_im_lm_j$ or $m_jm_lm_i$. Finally, we commute $m_i$ and $m_j$ using operation (iii) at a cost of $|m_i-m_j|$ and cancel $m_im_j$.
	
	Now we are ready to estimate the cost to cancel the sequence $m_1,m_2,\dots,m_{2k}$ to the empty sequence. By $(2)$, we can assume that we have already cancelled all the pairs $m_i,m_j$ where $\sigma(m_i,m_j)=0, m_i=m_j$ using operations (i) and (ii). This step costs nothing and does not change any $\sigma(m_i,m_j)$ for $m_i,m_j$ remaining in the resulting sequence. Let the remaining elements after cancellations be $m_{i(1)}, m_{i(2)},\dots, m_{i(4s)}$ for some $2s\leqslant k$ and $i(1)<i(2)<\dots<i(4s)$. The reason we will have even pairs numbers left is that if the number of pairs is odd, there must be a pair of number of the same value that is even positions apart. The remaining sequence satisfies the following properties: 
	\begin{enumerate}[(a)]
		\item $\sigma(m_{i(s)},m_{i(l)})=i(s)-i(l) \mod 2$,
		\item if $m_{i(s)}=m_{i(l)}$ then $\sigma(m_{i(s)},m_{i(l)})=1$,
		\item $\sigma(m_{i(s)},m_{i(l)})=\sigma(m_{i(s')},m_{i(l')})$ for $m_{i(s)}=m_{i(s')},m_{i(l)}=m_{i(l')}$.
	\end{enumerate}
	Here the property (a) is true because in the original sequence $\iota(m_{i(s)})=i(s)$ and we only use operation (i) and (ii) which do not change $\sigma(m_{i(s)},m_{i(l)})$. (b) and (c) follow from the definition of $\sigma$ and the remaining sequence.
	\begin{figure}[H]
		\centering
			\includegraphics[width=3.5cm]{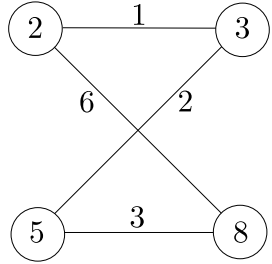}
			\caption{the corresponding graph of the sequence 2,3,5,3,5,8,2,8}
	\end{figure}
	We define the weighted graph $\Gamma_0$ associated with $(m_{i(1)}, m_{i(2)},\dots, m_{i(4s)})$ where the vertex set is $\{m_{i(1)}, m_{i(2)},\dots, m_{i(4s)}\}$ and there is an edge with weight $|m_{i(s)}-m_{i(l)}|$ connects $m_{i(s)},m_{i(l)}$ if $\sigma(m_{i(s),i(l)})=0$. Note that this graph is invariant under operations (ii) and may have multi-edge.
	
	By (3), we are allowed to cancel $m_{i(s)},m_{i(l)}$ at a cost of $|m_{i(s)}-m_{i(j)}|$ for some $m_{i(j)}$ that $\sigma(m_{i(s)},m_{i(j)})=0$. After the cancellation, since we use operation (iii) once, $\sigma(m_{i(j)},m_{i(j')})$ change to 0 for some $m_{i(k')}$ that $m_{i(j)}=m_{i(j')}$. Therefore we can then cancel $m_{i(j)},m_{i(j')}$ without any cost. In summary, we have
	\begin{enumerate}[(1)]
	  \setcounter{enumi}{3}
	  \item for $m_{i(s)} \neq m_{i(j)}$ that $\sigma(m_{i(s)},m_{i(j)})=0$, we can cancel a pair of number $m_{i(s)}$ and a pair of number $m_{i(j)}$ at a cost of $|m_{i(s)}-m_{i(j)}|$ where $\sigma(m_i,m_j)$ remains the same for numbers that have not been cancelled.
	\end{enumerate}
	(4) will delete an edge of $(m_{i(s)},m_{i(j)})$ in the graph. If no edges connecting $m_{i(s)}$ and $m_{i(j)}$, we delete the two vertices $m_{i(s)}, m_{i(j)}$. The cost is the weight of that edge. Let $\mathcal C$ be a cancellation of $\Gamma_0$ where $\mathcal C$ consists of an ordered sequence of edges in $\Gamma_0$, where we cancel the edge by the order of the sequence. Thus the total cost of a cancellation $\mathcal C$ to the empty graph is just a sum of the weight of edges in $\mathcal C$. Every cancellation can be associated with a path $p_\mathcal C$ where the path passes through the sequence of edges in $\mathcal C$ in the same order. 
	
	Now we delete edges in the following way. We first delete one edge $(m_{i(1)},m_{i(2)})$ since $\sigma(m_{i(1)},m_{i(2)})=0$. We let the resulted graph to be $\Gamma_1$. 
	\begin{figure}[H]
		\centering
			\includegraphics[width=10cm]{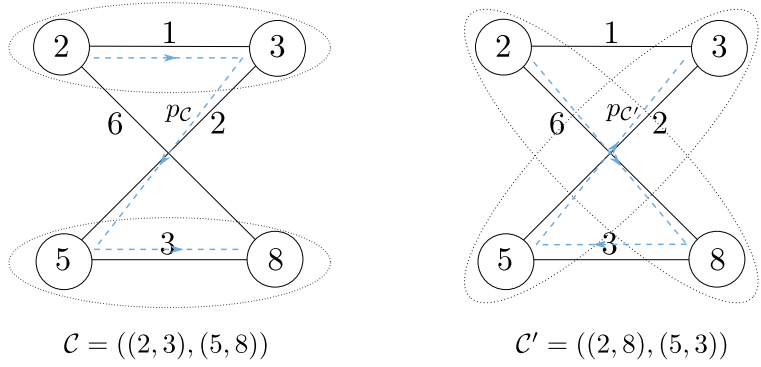}
			\caption{two different cancellations $\mathcal C,\mathcal C'$ and their corresponding $p_{\mathcal C},p_{\mathcal C'}$. The total cost of $\mathcal C$ is 3 and the total cost of $\mathcal C'$ is 8.}
	\end{figure}
	Inductively, $\Gamma_{i+1}$ is obtained by deleting an edge $(m_{i(s)},m_{i(j)})$ where $i,j$ are the smallest numbers remained in $\Gamma_i$. $\Gamma_s$ will be an empty graph since every time we delete four numbers from the sequence. 
	
	Let us estimate the cost from $\Gamma_0$ to $\Gamma_s$. Since every time we cancel pairs of numbers based on the order of the original sequence $m_1,m_2,\dots,m_{2k}$ (always cancel the first two numbers remained). The cost is bounded by 
	\[\sum_{i=1}^{2k-1} |m_{i+1}-m_i|=\sum_{i=2}^{2k} |n_i|<n.\]
	Let inequality can also be realized by the following interpretation: the sequence $m_{i(1)}, m_{i(2)},\dots, m_{i(4s)}$ defines a path $p$ in $\Gamma_0$ (since $\sigma(m_{i(i)},m_{i(j+1)})=0$) that $p(j)=m_{i(j)}$, the weight of the path $p$ is bounded by $n$ by the definition of $m_i$. $p$ happens to be the path associated with this cancellation. It follows that the cost of the cancellation is bounded by the total weight of $p$. Thus the total cost is bounded by $n$.
	
	By \cref{relativeCommutative}, the total cost of converting 
	\[a^{t^{m_1}}a^{t^{m_2}}\dots a^{t^{m_{2k}}}\]
	to $0$ is bounded by $4n-3$. We finish the proof.\end{proof}

\section{Relative Dehn Functions and Subgroup Distortions}
\label{relativeDehn5}

So far for all the examples considered in \cite{Fuh2000} and \cref{relativeDehn4}, only the upper bounds of their relative Dehn functions are estimated. Similar to the case of the Dehn function, it is genuinely much harder to estimate the lower bound. In this section, we will connect the relative Dehn function of a finitely generated metabelian group to the subgroup distortions in a wreath product of two free abelian groups. This connection provides a new method to estimate the lower bound for the relative Dehn function and yields a sequence of examples of finitely generated metabelian groups with relative Dehn function larger that $n^k$ for arbitrary $k\in \mathbb N$.

Let $G$ be a finitely generated group with a finite generating set $X$ and $H$ be a a subgroup of $G$ with finite generating set $Y$. The \emph{distortion function} of $H$ in $G$ is 
\[\Delta_H^G(n)=\sup\{|w|_Y\mid w\in H, |w|_X\leqslant n\}.\]

For example, the subgroup $\langle a \rangle$ in the Baumslag-Solitar group $\langle a,t\mid a^t=a\rangle$ is exponentially distorted since $a^{t^n}=a^{2^n}$. And it not hard to check that infinite subgroups of a finitely generated abelian group are undistorted. 

Let $A$ and $T$ be free abelian groups with bases $\{a_1,a_2,\dots,a_m\}$ and $\{t_1,t_2,\dots,t_k\}$ respectively. Consider the wreath product $W:=A\wr T$. The base group $B:=\llangle A\rrangle$ is a $T$-module. For a finite subset $\mathcal X=\{f_1,f_2,\dots,f_l\}$ of $B$, let $H$ be the subgroup of $W$ generated by $\mathcal X\cup\{t_1,t_2,\dots,t_k\}$ and $G$ be the group $W/\llangle \mathcal X\rrangle$. We denote by $\pi: W\twoheadrightarrow T$ the canonical quotient map. 

\begin{theorem}
	\label{subgroupDistortion}
	Let $W,H,G$ be groups defined as above, then 
	\[\Delta_{H}^W(n) \preccurlyeq \tilde \delta_G^k(n)+n^k,\tilde\delta_G(n)\preccurlyeq\max\{n^3, (\Delta_{H}^W(n^2))^3\}.\]
	In particular, if $k=1$, 
	\[\Delta_{H}^W(n) \preccurlyeq \tilde \delta_G(n).\]
\end{theorem}
	
\begin{proof}
	First we show the following lemma.
	
	\begin{lemma}
		let $M$ be the $T$-module $B/\llangle \mathcal X \rrangle$. Then $\hat\delta_M(n)\preccurlyeq\Delta_H^W(n)\preccurlyeq \hat\delta_M^k(n)+n^k$. 

	\end{lemma}
	
	\begin{proof}
	Let $g\in H$. Note that $g$ can be written as $g_0t$, by adding $t:=\pi(g)$ to the end, where $g_0\in B, t\in T$. Since $|\pi(t)|_T \leqslant |g|_W\leqslant n$, $|g_0|_W\leqslant 2|g|_W$. Thus, we have
	\[|g|_H=|g_0t|_H\leqslant |g_0|_H+|t|_H\leqslant |g|_W+|g_0|_H.\]
	Assume that the ordered form of $OF(g_0)$ is $a_1^{\mu_1} a_2^{\mu_2}\dots a_m^{\mu_m}$, let us estimate $|g_0|_H$. First note that $\deg \mu_i\leqslant |g|_W$ for all $i$. Let $\alpha_1,\alpha_2,\dots,\alpha_l$ be elements in $\mathbb Z T$ such that $g_0=f_1^{\alpha_1}f_2^{\alpha_2}\dots f_l^{\alpha_l}$ and $\sum_{i=1}^l |\alpha_i|$ is minimized. By Theorem 3.4 in \cite{davis2011Subgroup}, 
	\[|g_0|_H=\sum_{i=1}^l |\alpha_i|+\mathrm{reach}(g_0),\]
	where $\mathrm{reach}(g_0)$ is the length of the shortest loop starting at 0 in the Cayley graph of $T$ that passing through all points in the set $\cup_{i=1}^l \supp \alpha_i$. By \cite[Lemma 6.7]{wang2020dehn}, for all $i$, $\deg (\alpha_i)\leqslant |g|_W+C\sum_{i=1}^l |\alpha_i|$ for some constant $C$. It follows that $\cup_{i=1}^l \supp \alpha_i$ lies in Ball $B_0(|g|_W+C\sum_{i=1}^l |\alpha_i|)$ of radius $|g|_W+C\sum_{i=1}^l |\alpha_i|$ centered at 0 in the Cayley graph of $T$. Since there exists a path of length $(2(|g|_W+C\sum_{i=1}^l |\alpha_i|)+1)^k$ passing through all the points in $B_0(|g|_W+C\sum_{i=1}^l |\alpha_i|)$, 
	\[\mathrm{reach}(g_0) \leqslant (2(|g|_W+C\sum_{i=1}^l |\alpha_i|)+1)^k.\] 
	Therefore, we have
	\[\sum_{i=1}^l |\alpha_i|\leqslant |g|_H\leqslant |g|_W+\sum_{i=1}^l |\alpha_i|+2^k(|g|_W+C\sum_{i=1}^l |\alpha_i|)^k.\]
	Since $\sum_{i=1}^l |\alpha_i|=\widehat \area_M(g_0)$ by definition and $\|g_0\|\leqslant 2|g|_W$, we have the following estimation:
	\[\widehat \area(2|g|_W)\leqslant |g|_H \leqslant |g|_W+\widehat \area(2|g|_W)+2^k(|g|_W+C\widehat \area(2|g|_W))^k.\]
	\end{proof}
	
	By \cref{relativeConnection1}, we have
	\[\Delta_{H}^W(n) \preccurlyeq \tilde \delta_G^k(n)+n^k.\]
	Last, by \cref{improved}, 
	\[\tilde\delta_G(n)\preccurlyeq\max\{n^3, (\Delta_{H}^W(n^2))^3\}.\]
\end{proof}

\cref{subgroupDistortion} connects the subgroup distortion function and the relative Dehn function, as it provides a way to estimate the relative Dehn function from the bottom. One special case is that both $A$ and $T$ are free abelian group of rank 1. Davis and Olshanskiy \cite{davis2011Subgroup} show that subgroups in $W=\langle a\rangle \wr \langle t\rangle$ have polynomial distortion functions and moreover for each $l\in \mathbb Z$, a subgroup of the form $H_l:=\langle [\dots,[a,t],t],\dots,t],t \rangle$, where the commutator is $(l-1)$-fold, is isomorphic to $\mathbb Z\wr \mathbb Z$ with $n^l$ distortion. It follows immediately that 

\begin{corollary}
\label{relativeLowerbound}
	Let $W=\langle a\rangle \wr \langle t \rangle$ be the wreath product of two infinite cyclic group. For each $l\in \mathbb N$, let $w_l=[\dots,[a,t],t],\dots,t]$ be the $(l-1)$-fold commutator. Finally let $H_l=W/\llangle w_l\rrangle$. Then we have
	\[\tilde \delta_{H_l}\succcurlyeq n^l.\]
\end{corollary}

Let us consider the case when the rank of $T$ is 1, that is, when $k=1$. The following result characterizes the distortion function of subgroups when $k=1$.

\begin{theorem}[Davis, Olshanskiy, {\cite[Theorem 1.2]{davis2011Subgroup}}]
\label{subgroupoverZ}
	Let $A$ be a finitely generated abelian group.
	\begin{enumerate}[(1)]
		\item For any finitely generated infinite subgroup $H\leqslant A\wr \mathbb Z$ there exists $l\in \mathbb N$ such that the distortion of $H$ in $A\wr \mathbb Z$ is 
		\[\delta_{H}^{A\wr \mathbb Z}(n)\asymp n^l.\]
		\item If $A$ is finite, then $l=1$; that is, all subgroup are undistorted.
		\item If $A$ is infinite, then for every $l\in \mathbb N$, there is a 2-generated subnormal subgroup $H$ of $A\wr \mathbb Z$ having distortion function
		\[\Delta_H^{A\wr \mathbb Z}(n)\asymp n^l.\]
	\end{enumerate}	
\end{theorem}

It follows that

\begin{theorem}
	\label{overZ}
	Let $G$ be a finitely generated metabelian group such that $rk(G)=1$. Then the relative Dehn function of $G$ is polynomially bounded. If in addition $G$ is finitely presented, the Dehn function of $G$ is asymptotically bounded above by the exponential function.
\end{theorem}

\begin{proof}
	By passing to a finite index subgroup, we can assume that there exists a short exact sequence
	\[1\to A\to G\to \mathbb Z\to 1,\]
	where $A$ is abelian.

	We denote by $T=\langle t\rangle$ the $\mathbb Z$ in the short exact sequence. Since every short exact sequence $1\to A\to G\to \mathbb Z\to 1$ splits, $G$ is isomorphic to the semidirect product $A\rtimes T$.
	
	Note that $A$ is a normal subgroup of $G$, then it is finitely generated as a $T$-module. Thus, there exists a free $T$-module $M$ of rank $m$ and a submodule $S=\langle f_1,f_2,\dots,f_l\rangle$ such that $A\cong M/S$. We have that 
	\[G\cong (M/S)\rtimes T\cong (M\rtimes T)/\llangle f_1,f_2,\dots,f_l\rrangle.\] 
	Let $\bar A$ be a free abelian group of rank $m$ and $W:=\bar A \wr T$ be the wreath product of $\bar A$ and $T$. Then there is an isomorphism $\varphi: M\rtimes T\to W$. We have
	\[G\cong W/\llangle \varphi(f_1),\varphi(f_2),\dots,\varphi(f_l)\rrangle.\]
	
	Let $H$ be the subgroup in $W$ generated by $\{\varphi(f_1),\varphi(f_2),\dots,\varphi(f_l),t\}$. By \cref{subgroupDistortion}, we have that 
	\[\tilde \delta_G(n) \preccurlyeq\max\{n^3, (\Delta_{H}^W(n^2))^3\}.\]
	By \cref{subgroupoverZ}, $\Delta_{H}^W(n)$ is a polynomial. Therefore the relative Dehn funcion $\tilde\delta_G(n)$ of $G$ is polynomially bounded. We are done for the relative Dehn function case.
	
	If $G$ is finitely presented, by \cref{relativeConnection2}, if the relative Dehn function is polynomially bounded, $\delta_G(n)$ is bounded above by the exponential function.
\end{proof}

\medskip

\bibliography{MyLibrary}{}
\bibliographystyle{alpha}

\end{document}